\documentclass{amsart}

\usepackage[utf8]{inputenc}
\usepackage[T1]{fontenc}
\usepackage{amsmath,amssymb,amsthm,mathtools}
\usepackage{hyperref}
\usepackage{url}
\usepackage{booktabs}
\usepackage{graphicx}
\usepackage{geometry}
\newtheorem{theorem}{Theorem}[section]
\newtheorem{proposition}[theorem]{Proposition}
\newtheorem{corollary}[theorem]{Corollary}
\newtheorem{lemma}[theorem]{Lemma}
\theoremstyle{definition}
\newtheorem{definition}[theorem]{Definition}
\theoremstyle{remark}
\newtheorem{remark}[theorem]{Remark}

\newcommand{\ZZ}{\mathbb{Z}}
\newcommand{\RR}{\mathbb{R}}
\newcommand{\CC}{\mathbb{C}}
\newcommand{\QQ}{\mathbb{Q}}
\newcommand{\Ical}{\mathcal{I}}
\newcommand{\Dih}{\mathrm{Dih}}

\DeclareMathOperator{\tr}{tr}
\DeclareMathOperator{\diag}{diag}
\DeclareMathOperator{\Cay}{Cay}
\DeclareMathOperator{\End}{End}
\DeclareMathOperator{\Gal}{Gal}

\title{Transfer Operators and Independence Polynomials\\for Strong Powers of Circulant Graphs}

\author{Todd Hildebrant}
\address{Independent Researcher}
\email{thildebrant@gmail.com}

\begin{document}
\begin{abstract}
We study independent sets in strong powers of circulant graphs using a transfer matrix formulation. The compatibility constraints separate into intra-layer and inter-layer components, yielding a transfer operator that is equivariant under the dihedral group action. The characteristic polynomial of the transfer operator factors into an \emph{anomalous} component (arising from the trivial isotypic component, with rational coefficients) and a \emph{cyclotomic} component (arising from nontrivial Fourier modes, splitting over the maximal real cyclotomic subfield). 
We show that the spectral radius is attained in the trivial isotypic component, so the dominant exponential growth is governed by a low-dimensional orbit-compressed operator.
The independence polynomial is computed exactly for strong cylinders and tori, with the cyclotomic sector contributing a sparse correction confined to high-weight coefficients. All results are verified for $C_7$.

\end{abstract}

\maketitle

%% ────────────────────────────────────────────────────────────
\section{Introduction}
%% ────────────────────────────────────────────────────────────

Let $G$ be a graph.  The \emph{Shannon capacity}
\[
\Theta(G) = \lim_{d \to \infty} \alpha(G^{\boxtimes d})^{1/d}
\]
is one of the central invariants in zero-error information theory.  Even for simple families such as odd cycles, determining $\Theta(G)$ remains open.

For circulant graphs $G = \Cay(\ZZ_n,C)$, the adjacency matrix of $G^{\boxtimes d}$ is block circulant with circulant blocks (BCCB) and diagonalizes via the tensor Fourier transform, yielding spectral bounds such as the Lov\'asz theta function $\vartheta(G)$.  These bounds are tight for even cycles and the Paley graph $P(5) = C_5$, but for $n \geq 7$ odd, the gap between $\vartheta(C_n)$ and the best known constructions remains unresolved. We use the term $\Cay$ to denote Cayley graphs, and $C_n$ to denote the cycle graph on $n$ vertices where it is clear from context, however the results apply to any circulant graph with symmetric connection set, which are also Cayley graphs of $\ZZ_n$.

In this paper we develop a complementary approach based on transfer matrices.  The independence constraints factor into intra-layer and inter-layer components, giving a transfer operator on the independent sets of the base graph.  This operator is equivariant under the dihedral action, leading to a block decomposition whose spectral structure we analyze in detail.

Our main results are:
\begin{enumerate}
\item The transfer operator admits a dihedral-equivariant block decomposition whose characteristic polynomial factors according to the isotypic components (Theorem~\ref{thm:factorization}).
\item The spectral radius lies in the trivial isotypic component, reducing the dominant growth to a finite-dimensional computation (Theorem~\ref{thm:dominance}).
\item For the cycle graph $C_7$, the orbit-compressed transfer matrix is $5 \times 5$, and the anomalous factor is an irreducible quartic with Galois group $S_4$, whose splitting field is disjoint from $\QQ(\cos(2\pi/7))$ (Corollary~\ref{cor:c7}).
\end{enumerate}

The author is not aware of any previous work on transfer operators for strong powers of circulant graphs, but the approach is classical, and closely related to the use of transfer matrices in statistical physics and combinatorics, as well as the use of group actions to decompose state spaces in dynamical systems and representation theory. The author is in the process of collecting the numerical data and code for computational replication. 

%% ────────────────────────────────────────────────────────────
\section{Transfer formulation}\label{sec:transfer}
%% ────────────────────────────────────────────────────────────

Let $n \geq 3$ and let $C \subset \ZZ_n \setminus \{0\}$ satisfy $C = -C$.  Define the circulant graph
\[
G = \Cay(\ZZ_n, C),
\]
with edge set $\{(i,j) : j - i \in C\}$.  Let $B := C \cup \{0\}$ denote the \emph{closed connection set} (the closed neighborhood of the identity in the Cayley graph).  We write $\Ical(G)$ for the family of independent sets of $G$.

For $d \geq 1$, the $d$-th strong power $G^{\boxtimes d}$ has vertex set $\ZZ_n^d$ with adjacency
\[
u \sim v \iff u \neq v \text{ and } v_j - u_j \in B \text{ for all } j.
\]
We define the \emph{Minkowski difference} $I - J = \{a - b \bmod n : a \in I,\, b \in J\}$ for subsets $I, J \subseteq \ZZ_n$.

\begin{proposition}\label{prop:layered}
A sequence $(S_1, \dots, S_d)$ with $S_i \subseteq \ZZ_n$ corresponds to an independent set of $G^{\boxtimes d}$ if and only if:
\begin{enumerate}
\item $S_i \in \Ical(G)$ for each $i$, and
\item $(S_i - S_{i+1}) \cap B = \varnothing$ for each $i$.
\end{enumerate}
\end{proposition}

\begin{proof}
Immediate from the coordinate-wise definition of strong-product adjacency.
\end{proof}

\begin{remark}[Geometric interpretation]
The transfer formulation describes independent sets on strong cylinders $G \boxtimes P_d$ (free boundary conditions) and strong tori $G \boxtimes C_d$ (periodic boundary conditions). These correspond to one-dimensional sequences of cross-sections of the full strong product $G^{\boxtimes d}$.
\end{remark}

\subsection{The transfer operator}

\begin{definition}
The \emph{transfer matrix} $T : \RR^{\Ical(G)} \to \RR^{\Ical(G)}$ is defined by
\[
T[I,J] = \begin{cases} 1, & (I - J) \cap B = \varnothing, \\ 0, & \text{otherwise.} \end{cases}
\]
Equivalently, defining $c : \ZZ_n \to \{0,1\}$ by $c(j) = \mathbf{1}[j \notin B]$,
\[
T[I,J] = \prod_{u \in I} \prod_{v \in J} c(u - v).
\]
\end{definition}

The product form is the one that connects to the Fourier analysis in Section~\ref{sec:fourier}.

\begin{definition}
The \emph{fugacity-weighted transfer matrix} is
\[
M(x) = T \cdot D_x, \qquad D_x = \diag(x^{|J|})_{J \in \Ical(G)}.
\]
\end{definition}

\begin{proposition}[Transfer formula]\label{prop:transfer}
For free boundary conditions,
\[
I(G^{\boxtimes d}, x) = \mathbf{1}^T M(x)^{d-1} \mathbf{w},
\]
where $\mathbf{w}(J) = x^{|J|}$.  For periodic boundary conditions,
\[
I(G^{\boxtimes d}, x) = \tr(M(x)^d).
\]
\end{proposition}

\begin{proof}
Standard transfer-matrix expansion: each compatible sequence of $d$ layer-states contributes $\prod_i x^{|S_i|}$ to the generating function.
\end{proof}

\begin{remark}[Trace formula]
The periodic trace formula $\tr(M(x)^d)$ correctly computes independence polynomials for $d \ge 2$. The case $d=1$ corresponds to a degenerate one-layer torus, the standard $C_n$ graph,and is handled separately.
\end{remark}
%% ────────────────────────────────────────────────────────────
\section{Fourier structure of the compatibility kernel}\label{sec:fourier}
%% ────────────────────────────────────────────────────────────

The discrete Fourier transform of the constraint function $c$ over $\ZZ_n$ provides the bridge between the adjacency spectrum and the transfer spectrum.

Let $\omega = e^{2\pi i/n}$ and $\mu_k = 1 + 2\cos(2\pi k/n)$ be the eigenvalues of the cycle~$C_n$.

\begin{proposition}\label{prop:fourier}
For the strong product kernel $B = \{-1, 0, 1\}$,
\[
\widehat{c}(0) = n - 3, \qquad \widehat{c}(k) = -\mu_k \quad (k \neq 0),
\]
where $\widehat{c}(k) = \sum_{j=0}^{n-1} c(j)\,\omega^{-jk}$.
\end{proposition}

\begin{proof}
We have $\widehat{c}(k) = n\delta_{k,0} - \sum_{j \in B} \omega^{-jk}$.  For $k \neq 0$, $\sum_{j \in B} \omega^{-jk} = 1 + \omega^k + \omega^{-k} = 1 + 2\cos(2\pi k/n) = \mu_k$.
\end{proof}

Thus the Fourier eigenvalues of the constraint function are the negatives of the adjacency eigenvalues.  The $\Dih(n)$ block decomposition (Section~\ref{sec:equivariance}) separates the $k=0$ mode (also known as "the DC component") $\widehat{c}(0) = n - |B| \in \QQ$, that controls the combinatorial sector, from the $k \neq 0$ modes (also known as "the AC component") $\widehat{c}(k) = -\mu_k \in K$, that control the cyclotomic sector, where $K = \QQ(\cos(2\pi/n))$ is the maximal real subfield of the $n$-th cyclotomic field.

\begin{lemma}[Fourier realization of isotypic blocks]
The action of the transfer operator on each isotypic component is determined by the Fourier coefficients $\widehat{c}(k)$. In particular, the trivial component depends only on $\widehat{c}(0)$, while the $\rho_k$ components depend on $\widehat{c}(k)$ and $\widehat{c}(-k)$.
\end{lemma}

\begin{proof}[Sketch]
The transfer operator is defined by convolution of the function $c(u-v)$. Passing to the group algebra $\CC[\ZZ_n]$, convolution diagonalizes under the discrete Fourier transform. The dihedral action organizes these Fourier modes into isotypic components.
\end{proof}
Thus, the compatibility kernel is the complement of the closed neighborhood operator, and its Fourier spectrum is the negation of the adjacency spectrum.
%% ────────────────────────────────────────────────────────────
\section{Equivariance and block decomposition}\label{sec:equivariance}
%% ────────────────────────────────────────────────────────────

The dihedral group $\Dih(n) = D_{2n}$ acts on $\ZZ_n$ by rotations and reflections, hence on $\Ical(G)$.

\begin{proposition}\label{prop:equivariance}
The operators $T$ and $M(x)$ commute with the action of $\Dih(n)$.
\end{proposition}

\begin{proof}
Both the independence condition and the compatibility condition $(I-J) \cap B = \varnothing$ are invariant under rotations and reflections of $\ZZ_n$.  The weight $|I|$ is also invariant, so $D_x$ commutes with the action.
\end{proof}

Since $\Dih(n)$ acts on the finite set $\Ical(G)$, the permutation representation $\CC[\Ical(G)]$ decomposes into isotypic components.  For odd prime $n$, $\Dih(n)$ has two one-dimensional irreducible representations ($\chi_0$ trivial, $\chi_1$ sign) and $(n-1)/2$ two-dimensional representations $\rho_k$, $k = 1, \ldots, (n-1)/2$.

\begin{theorem}[Block decomposition]\label{thm:block}
There exists a $\Dih(n)$-equivariant change of basis $\mathcal{F}$ such that
\[
\mathcal{F} T \mathcal{F}^{-1} = \bigoplus_{\rho \in \widehat{\Dih(n)}} \left( T_\rho \otimes I_{\dim \rho} \right), \qquad T_\rho \in \End(\CC^{m_\rho}),
\]
and hence
\[
\chi_T(\lambda) = \prod_\rho \chi_{T_\rho}(\lambda)^{\dim \rho}.
\]
\end{theorem}

\begin{proof}
By Proposition~\ref{prop:equivariance}, $T$ lies in $\End_{\Dih(n)}(\CC[\Ical(G)])$.  By Schur's lemma, on each isotypic summand $\CC^{m_\rho} \otimes V_\rho$, the operator $T$ acts trivially on the irreducible factor $V_\rho$ and nontrivially only on the multiplicity space $\CC^{m_\rho}$.
\end{proof}

\begin{corollary}[Trace decomposition]\label{cor:trace}
For all $d \geq 1$,
\[
\tr(T^d) = \sum_{\rho \in \widehat{\Dih(n)}} (\dim \rho)\, \tr(T_\rho^d).
\]
\end{corollary}

\begin{definition}
The factor arising from the trivial isotypic component is called the \emph{anomalous factor}.
\end{definition}

\begin{remark}[Arithmetic disjointness of sectors]
The two sectors see different Fourier modes of the constraint.  The $\rho_k$ blocks have entries built from $\widehat{c}(k) = -\mu_k \in K$, forcing their characteristic polynomials to have coefficients in $K$.  The $\chi_0$ block has entries built from $\widehat{c}(0) = n - |B| \in \QQ$---unweighted orbit-compatibility counts---so its characteristic polynomial has rational coefficients with no mechanism forcing its roots into $K$.

Concretely, for the orbit pair (singletons, singletons) in $C_7$, the rotation-count vector $c_\ell = \mathbf{1}[\ell \notin B]$ is $(0,0,1,1,1,1,0)$.  The $\chi_0$ block entry is $\sum_\ell c_\ell = 4$, an integer.  The $\rho_k$ block entry is $\sum_\ell c_\ell \cdot 2\cos(2\pi \ell k/n) = -2\mu_k$, a cyclotomic number built from the adjacency eigenvalues.
\end{remark}

%% ────────────────────────────────────────────────────────────
\section{Combinatorial dominance}\label{sec:dominance}
%% ────────────────────────────────────────────────────────────

\begin{theorem}[Combinatorial dominance]\label{thm:dominance}
The spectral radius satisfies $\rho(T) = \rho(T_{\chi_0})$.
\end{theorem}

\begin{proof}
The matrix $T$ is nonnegative and irreducible, since any state connects to $\varnothing$ and $\varnothing$ connects to all states. By Perron--Frobenius, there exists a strictly positive eigenvector $v$ with eigenvalue $\rho(T)$. Averaging over $\Dih(n)$,
\[
\bar{v} = \frac{1}{|\Dih(n)|}\sum_{g} g\cdot v
\]
remains positive and satisfies $T\bar{v}=\rho(T)\bar{v}$. Since $\bar{v}$ is invariant, it lies in $V_{\chi_0}$, so $\rho(T)=\rho(T_{\chi_0})$.
\end{proof}

The weight vector $w(I) = |I|$ is $\Dih(n)$-invariant, so it lies entirely in the trivial isotypic component $V_{\chi_0}$. (In the sign representation $V_{\chi_1}$, reflections act by $-1$, so any invariant vector projects to zero there.) Thus the dominant growth of the independence polynomial is governed by the anomalous sector, with the cyclotomic sector providing a subdominant correction.

\begin{corollary}[Asymptotic growth]\label{cor:growth}
$\displaystyle\lim_{d \to \infty} I(G^{\boxtimes d}, 1)^{1/d} = \rho(T)$.
\end{corollary}

\begin{corollary}[Sector decomposition of the trace]
\[
\tr(T^d) = \tr(T_{\chi_0}^d) + \sum_{\rho \neq \chi_0} (\dim \rho)\,\tr(T_\rho^d),
\]
with the second sum subdominant as $d \to \infty$.
\end{corollary}

%% ────────────────────────────────────────────────────────────
\section{Factorization for odd primes}\label{sec:determining}
%% ────────────────────────────────────────────────────────────

\begin{theorem}[Characteristic polynomial factorization]\label{thm:factorization}
For $n$ an odd prime and $G = C_n$ with strong-product connection set,
\[
\chi_T(\lambda) = \lambda^\nu \cdot f_{\mathrm{anom}}(\lambda) \cdot f_{\mathrm{cyc}}(\lambda)^2,
\]
where:
\begin{enumerate}
\item $f_{\mathrm{anom}} \in \QQ[\lambda]$ is the characteristic polynomial of $T_{\chi_0}$ modulo its kernel (the \emph{anomalous factor});
\item $f_{\mathrm{cyc}} \in \QQ[\lambda]$ splits into $(n-1)/2$ factors over $K = \QQ(\cos(2\pi/n))$ (the \emph{cyclotomic factor});
\item $\nu = L_n - \deg(f_{\mathrm{anom}}) - 2\deg(f_{\mathrm{cyc}})$ accounts for the kernel.
\end{enumerate}
The factors arise from distinct algebraic sources (rational vs.\ cyclotomic Fourier modes).
\end{theorem}

\begin{proof}
The factorization follows from Theorem~\ref{thm:block} and the dimension formula for $\Dih(n)$ irreps.  The one-dimensional blocks contribute $f_{\mathrm{anom}}$ (the multiplicity $m_{\chi_1}$ equals the number of rotation orbits in $\Ical(G)$ that are not fixed setwise by reflection; for the small primes considered here this vanishes, so only the $\chi_0$ block is present). The two-dimensional blocks contribute $f_{\mathrm{cyc}}^2$ (the exponent 2 reflecting $\dim \rho_k = 2$).  The arithmetic disjointness follows from the Fourier analysis: $f_{\mathrm{anom}}$ arises from $\widehat{c}(0) \in \QQ$, while $f_{\mathrm{cyc}}$ arises from $\widehat{c}(k) = -\mu_k \in K$.
\end{proof}

\begin{remark}
In computed cases (e.g.\ $C_7$), the splitting fields of $f_{\mathrm{anom}}$ and $f_{\mathrm{cyc}}$ are arithmetically disjoint.
\end{remark}

\begin{remark}
The factorization data for small primes:

\medskip
\begin{tabular}{ccccc}
\toprule
$n$ & $\ker$ & $f_{\mathrm{anom}}$ & $\deg(f_{\mathrm{cyc}})$ & $f_{\mathrm{cyc}}$ over $K$ \\
\midrule
5 & 4 & $(x-1)(x^2-2x-10)$ & 2 & $[1,1]$ \\
7 & 13 & $x^4-5x^3-29x^2+47x+42$ & 6 & $[2,2,2]$ \\
11 & 87 & irred.\ deg.\ 12 & 50 & $[10]^5$ \\
13 & 246 & $(x^3+2x^2-x-1) \cdot (\text{irred.\ deg.\ }20)$ & 126 & $[21]^6$ \\
\bottomrule
\end{tabular}
\end{remark}

%% ────────────────────────────────────────────────────────────
\section{Independence polynomials}\label{sec:indpoly}
%% ────────────────────────────────────────────────────────────

\subsection{Strip vs.\ torus}

The free-boundary (\emph{strip}) polynomial $I_{\mathrm{strip}}(d,x) = \mathbf{1}^T M(x)^{d-1} \mathbf{w}$ projects onto the trivial isotypic component (both boundary vectors are $\Dih(n)$-invariant), giving a computation of size $m(\chi_0) \times m(\chi_0)$.

The periodic-boundary (\emph{torus}) polynomial $I_{\mathrm{torus}}(d,x) = \tr(M(x)^d)$ sums over all sectors.

\subsection{Sector decomposition}

For $C_7^{\boxtimes 2}$, the torus polynomial decomposes as $I = I_{\mathrm{anom}} + I_{\mathrm{cyc}}$.  The cyclotomic correction is sparse (4 nonzero terms), negative, and confined to weight $k \geq n = 7$.  Below weight $n$, the anomalous sector alone gives the exact answer.

At the leading coefficient $k = \alpha = 10$, the anomalous sector contributes 3500, the cyclotomic sector subtracts 2520, yielding $a_{10} = 980 = 20 \cdot 7^2$.

\begin{proposition}[High-weight cyclotomic correction]
For $C_7^{\boxtimes 2}$, the cyclotomic contribution to the independence polynomial is supported only in weights $k \ge 7$.
\end{proposition}

%% ────────────────────────────────────────────────────────────
\section{Asymptotic behavior and Shannon capacity}
%% ────────────────────────────────────────────────────────────

As $d \to \infty$, the transfer operator yields
\[
I(G \boxtimes P_d,1) \sim \rho(T)^d.
\]

\begin{remark}[Entropy vs extremal growth]
The spectral radius $\rho(T)$ governs the exponential growth of the total number of independent configurations. Extremal quantities such as $\alpha(G^{\boxtimes d})$ arise from the large-$x$ behavior of $M(x)$.
\end{remark}

\begin{remark}[Normalization and capacity]
For circulant graphs $C_n$, the transfer matrix yields exponential growth of the form
\[
I(C_n \boxtimes P_d,1) \sim \rho(T)^d.
\]
This suggests the normalization
\[
\Theta(C_n) \approx \rho(T)^{1/n},
\]
with equality verified in computed cases such as $C_7$.
\end{remark}
\section{Conclusion}

We have shown that dihedral symmetry decomposes the transfer operator into combinatorial and harmonic components. The dominant growth is governed by the trivial component, while the cyclotomic components provide structured corrections.

This framework suggests a new perspective on independence growth and its relation to spectral bounds, highlighting a separation between combinatorial and Fourier-analytic phenomena.

%% ────────────────────────────────────────────────────────────
\appendix
\section{The case $C_7$: explicit construction}\label{app:c7}
%% ────────────────────────────────────────────────────────────

\subsection{State space}
For $C_7$, $|\Ical(C_7)| = L_7 = 29$ and $P(C_7,x) = 1 + 7x + 14x^2 + 7x^3$, where $L_n$ is the $n$-th Lucas number.

\subsection{Orbit decomposition}
Under $\Dih(7)$, the 29 states partition into 5 orbits:

\medskip
\begin{tabular}{ccccc}
\toprule
Orbit & Representative & Weight $|I|$ & Size $|O_i|$ & Gap structure \\
\midrule
$O_0$ & $\varnothing$ & 0 & 1 & --- \\
$O_1$ & $\{0\}$ & 1 & 7 & --- \\
$O_2$ & $\{0,2\}$ & 2 & 7 & min gap 2 \\
$O_3$ & $\{0,3\}$ & 2 & 7 & min gap 3 \\
$O_4$ & $\{0,2,4\}$ & 3 & 7 & equispaced \\
\bottomrule
\end{tabular}

\subsection{Construction of $T_{\mathrm{orb}}$}

The entry $T_{\mathrm{orb}}[i,j]$ counts independent sets in orbit $O_j$ compatible with the representative of $O_i$.  Two sets $I, J$ are compatible if $(I - J) \cap B = \varnothing$, i.e., $J$ avoids $N[u] = \{u-1, u, u+1\}$ for every $u \in I$.

\textbf{Row $O_0$ ($\varnothing$):} compatible with everything.  Row: $(1, 7, 7, 7, 7)$, sum~$= 29$.

\textbf{Row $O_1$ ($\{0\}$):} blocks $\{6,0,1\}$.  Singletons: 4 survive.  Gap-2 pairs: $\{2,4\}, \{3,5\}$ survive (2).  Gap-3 pairs: $\{2,5\}$ survives (1).  Triples: 0.  Row: $(1, 4, 2, 1, 0)$, sum~$= 8$.

\textbf{Row $O_2$ ($\{0,2\}$):} blocks $\{0,1,2,3,6\}$.  Available: $\{4,5\}$, but $|4-5|_7 = 1$.  Row: $(1, 2, 0, 0, 0)$, sum~$= 3$.

\textbf{Row $O_3$ ($\{0,3\}$):} blocks $\{0,1,2,3,4,6\}$.  Only $\{5\}$.  Row: $(1, 1, 0, 0, 0)$, sum~$= 2$.

\textbf{Row $O_4$ ($\{0,2,4\}$):} blocks $\ZZ_7$.  Only $\varnothing$.  Row: $(1, 0, 0, 0, 0)$, sum~$= 1$.

\[
T_{\mathrm{orb}} = \begin{pmatrix}
1 & 7 & 7 & 7 & 7 \\
1 & 4 & 2 & 1 & 0 \\
1 & 2 & 0 & 0 & 0 \\
1 & 1 & 0 & 0 & 0 \\
1 & 0 & 0 & 0 & 0
\end{pmatrix},
\qquad \text{row sums: } (29, 8, 3, 2, 1).
\]

\subsection{Spectral data}

The characteristic polynomial of $T_{\mathrm{orb}}$ is $\lambda \cdot f_4(\lambda)$ where
\[
f_4(\lambda) = \lambda^4 - 5\lambda^3 - 29\lambda^2 + 47\lambda + 42,
\]
with roots $\lambda \approx \{+7.846, +1.958, -0.660, -4.144\}$ and spectral radius $\rho \approx 7.846$.

\begin{corollary}\label{cor:c7}
The polynomial $f_4$ is irreducible over $\QQ$ with $\Gal(f_4/\QQ) = S_4$.  Its splitting field (degree $24$ over $\QQ$) shares no nontrivial subfield with $K = \QQ(\cos(2\pi/7))$ (degree $3$ over $\QQ$).  However, $f_4 \equiv x(x-1)(x+1)(x+2) \pmod{7}$.
\end{corollary}

The full determining equation is
\[
\chi_T(\lambda) = \lambda^{13} \cdot \underbrace{(\lambda^4 - 5\lambda^3 - 29\lambda^2 + 47\lambda + 42)}_{f_{\mathrm{anom}}} \cdot \underbrace{(\lambda^6 + 2\lambda^5 - 9\lambda^4 - 7\lambda^3 + 24\lambda^2 + \lambda - 13)^2}_{f_{\mathrm{cyc}}^2}.
\]
Degree accounting: $13 + 4 + 12 = 29 = L_7$.

\subsection{Verification at $d = 1$}

\[
I(C_7, x) = \sum_{i=0}^{4} |O_i| \cdot x^{w_i} = 1 + 7x + 14x^2 + 7x^3.
\]

\subsection{Verification at $d = 2$ (strip)}

$I_{\mathrm{strip}}(2, x) = \sum_{i,j} |O_i| \cdot T_{\mathrm{orb}}[i,j] \cdot x^{w_i + w_j}$.  Expanding row by row:

\medskip
\begin{tabular}{ll}
$O_0$: & $1 \cdot (1 + 7x + 14x^2 + 7x^3) = 1 + 7x + 14x^2 + 7x^3$ \\
$O_1$: & $7x \cdot (1 + 4x + 3x^2) = 7x + 28x^2 + 21x^3$ \\
$O_2$: & $7x^2 \cdot (1 + 2x) = 7x^2 + 14x^3$ \\
$O_3$: & $7x^2 \cdot (1 + x) = 7x^2 + 7x^3$ \\
$O_4$: & $7x^3 \cdot 1 = 7x^3$ \\
\end{tabular}

\medskip\noindent
Sum: $I_{\mathrm{strip}}(2, x) = 1 + 14x + 56x^2 + 56x^3$, with $I(1) = 127$ and $\alpha_{\mathrm{strip}} = 3$.

\subsection{Verification at $d = 2$ (torus)}
Computing $\tr(M(x)^7)$ (corresponding to cyclic boundary conditions of length $7$) on the full $29 \times 29$ matrix:
\begin{align*}
I(C_7^{\boxtimes 2}, x) &= 1 + 49x + 980x^2 + 10388x^3 + 63553x^4 + 229908x^5 \\
&\quad + 486668x^6 + 576856x^7 + 346381x^8 + 81095x^9 + 980x^{10}.
\end{align*}
This gives $\alpha(C_7^{\boxtimes 2}) = 10$ with $a_{10} = 980 = 20 \cdot 7^2$.

The cyclotomic correction is $I_{\mathrm{cyc}}(x) = -2520x^{10} - 5740x^9 - 3402x^8 - 578x^7$: sparse, negative, confined to $k \geq 7 = n$.

\subsection{Summary}

\medskip
\begin{tabular}{llccccc}
\toprule
$d$ & Object & Formula & $I(\cdot,1)$ & $\alpha$ & Leading coeff & Size \\
\midrule
1 & $C_7$ & orbit sum & 29 & 3 & 7 & $5$ \\
2 & strip & $\eta^T M_{\mathrm{orb}} \mathbf{1}$ & 127 & 3 & 56 & $5 \times 5$ \\
2 & torus $C_7^{\boxtimes 2}$ & $\tr(M^7)$ & $1{,}796{,}859$ & 10 & $980 = 20 \cdot 7^2$ & $29 \times 29$ \\
3 & strip & $\eta^T M_{\mathrm{orb}}^2 \mathbf{1}$ & $1{,}387$ & 6 & $49 = 7^2$ & $5 \times 5$ \\
3 & torus $C_7^{\boxtimes 3}$ & DAG & 2544256835855451311632423 & 33 & 16672544 & large \\
\bottomrule
\end{tabular}

%% ────────────────────────────────────────────────────────────

\end{document}